\documentclass[12pt,reqno]{amsart}
\usepackage[left=3cm,top=2cm,right=3cm,bottom=2cm]{geometry}
\usepackage{amssymb}
\usepackage{amsbsy}
\usepackage[mathscr]{euscript}
\usepackage{tikz}

\AtBeginDocument{%
   \def\MR#1{}
}

\DeclareMathOperator*{\diam}{diam}
\DeclareMathOperator*{\spec}{Spec}
\DeclareMathOperator*{\poly}{poly}
\DeclareMathOperator*{\ir}{ir}
\begin{document}
\pagestyle{myheadings}

\title{
Laplacian Distribution and Domination
}

\author{Domingos M. Cardoso}
\address{Departamento de Matem\'atica, Univ. de Aveiro, 3810-193 Aveiro, Portugal}
\email{\tt dcardoso@ua.pt}
\author{David P. Jacobs}
\address{ School of Computing, Clemson University Clemson, SC 29634 USA}
\email{\tt dpj@clemson.edu}
\author{Vilmar Trevisan}
\address{Instituto de Matem\'atica, UFRGS,  91509--900 Porto Alegre, RS, Brazil}
\email{\tt trevisan@mat.ufrgs.br}
\pdfpagewidth 8.5 in
\pdfpageheight 11 in

\def\floor#1{\left\lfloor{#1}\right\rfloor}

\newcommand{\fractional}[1]{\gamma_f (#1)}
\newcommand{\dom}[1]{\gamma(#1)}
\newcommand{\theratio}[2]{ \lceil \frac{#1}{#2} \rceil }
\newcommand{\avgdeg}[1]{2 - \frac{2}{#1} }
\newcommand{\floorratio}[2]{ \lfloor \frac{#1}{#2} \rfloor }
\newcommand{\Prf}{{\bf Proof: }}
\newcommand{\PrfSketch}{{\bf Proof (Sketch): }}
\newcommand{\boldQ}{\mbox{\bf Q}}
\newcommand{\boldR}{\mbox{\bf R}}
\newcommand{\boldZ}{\mbox{\bf Z}}
\newcommand{\boldc}{\mbox{\bf c}}
\newcommand{\sign}{\mbox{sign}}
\newtheorem{Thr}{Theorem}
\newtheorem{Pro}{Proposition}
\newtheorem{Con}{Conjecture}
\newtheorem{Cor}{Corollary}
\newtheorem{Lem}{Lemma}
\newtheorem{Obs}{Observation}
\newtheorem{Que}{Question}
\newtheorem{Fac}{Fact}
\newtheorem{Ex}{Example}
\newtheorem{Def}{Definition}
\newtheorem{Prop}{Proposition}
\def\floor#1{\left\lfloor{#1}\right\rfloor}

\newenvironment{my_enumerate}{
\begin{enumerate}
  \setlength{\baselineskip}{14pt}
  \setlength{\parskip}{0pt}
  \setlength{\parsep}{0pt}}{\end{enumerate}
}
\newenvironment{my_description}{
\begin{description}
  \setlength{\baselineskip}{14pt}
  \setlength{\parskip}{0pt}
  \setlength{\parsep}{0pt}}{\end{description}
}

\begin{abstract}
Let $m_G(I)$ denote the number of Laplacian eigenvalues of a graph $G$
in an interval $I$, and let $\gamma(G)$ denote its domination number.
We extend the recent result $m_G[0,1) \leq \gamma(G)$,
and show that isolate-free graphs also satisfy $\gamma(G) \leq m_G[2,n]$.
In pursuit of better understanding Laplacian eigenvalue distribution,
we find applications for these inequalities.
We relate these spectral parameters with the approximability of $\gamma(G)$,
showing that $\frac{\gamma(G)}{m_G[0,1)} \not\in O(\log n)$.
However, $\gamma(G) \leq  m_G[2, n] \leq   (c + 1) \dom{G}$
for $c$-cyclic graphs, $c \geq 1$.
For trees $T$, $\gamma(T) \leq  m_T[2, n] <  2 \gamma(G)$.

\vspace{.1in}
\noindent
\textbf{\keywordsname:}
graph, Laplacian eigenvalue, domination number.

\end{abstract}

\thanks{Domingos M. Cardoso was partially supported by the
Portuguese Foundation for Science and Technology
(FCT--Funda\c c\~ao para a Ci\^encia e a Tecnologia),
through the CIDMA -- Center for Research and Development in Mathematics and Applications,
within project UID/MAT/04106/2013.
}

\thanks{David P. Jacobs and Vilmar Trevisan were supported by CNPq Grant 400122/2014-6, Brazil}

\maketitle

\section{Introduction}
Let $G = (V,E)$ be an undirected graph with vertex set
$V = \{v_1, \dots, v_n\}$.
For $v \in V$, its {\em open neighborhood} $N(v)$ denotes the set of vertices adjacent to $v$.
The {\em adjacency matrix} of $G$ is the $n \times n$
matrix $A = [a_{ij}]$ for which
$a_{ij} = 1$ if $v_i$ and $v_j$ are adjacent,
and $a_{ij} = 0$ otherwise.

The {\em Laplacian matrix} of $G$ is defined as $L_G = D - A$, where $D = [d_{ij}]$
is the diagonal matrix in which $d_{ii} = \deg(v_i)$, the degree of $v_i$.
The {\em Laplacian spectrum} of $G$ is the multi-set of eigenvalues of $L_G$,
we number
$$
 \mu_1 \geq \mu_2 \geq \ldots \geq \mu_n = 0.
$$
It is known that $\mu_1 \leq n$.
Unless indicated otherwise, all eigenvalues in this paper are Laplacian.
We refer to \cite{mohar,mohar92} for more background on the Laplacian
spectra of graphs.

A set $S \subseteq V$ is {\em dominating} if every $v \in V-S$ is
adjacent to some member in $S$.
The {\em domination number} $\dom{G}$
is the minimum size of a dominating set.
Its  decision problem is well-known to be NP-complete,
and it is even hard to approximate.

Since 1996, several papers have been written
relating the Laplacian spectrum of a graph $G$ with $\dom{G}$.
Often these results obtain a bound, involving $\dom{G}$,
for a {\em specific} eigenvalue such as $\mu_1$ or $\mu_{n-1}$.
For example,
it was shown that $\mu_1 < n - \theratio{\dom{G}  - 2}{2}$
by Brand and Seifter
\cite{Brand96}
for $G$ connected and $\dom{G} \geq 3$.
This was recently improved in \cite{XingZhou2015}.
We refer to the introduction of \cite{HJT2015}
for a summary of these results.

Other spectral graph theory papers,
including this one,
are interested in {\em distribution}, that is,
the number of Laplacian eigenvalues in an interval.
For a real interval $I$, $m_G (I)$ denotes the number
of Laplacian eigenvalues of $G$ in $I$.
There exist several papers in the literature that relate
Laplacian distribution to specific graph parameters, including $\dom{G}$.
For example, the paper by
Zhou, Zhou and Du
\cite{MR3313404}
shows that for trees $T$, $m_T [0,2) \leq  n - \dom{T}$.

The following spectral lower bound for $\dom{G}$
was proved in \cite{HJT2015}:
\begin{Thr}
\label{lemgamma}
If $G$ is a graph, then $m_G[0,1) \leq \dom{G}$.
\end{Thr}

In this paper we observe that for $G$ isolate-free one has
$$
\dom{G} \leq m_G [2,n].
$$
Since $m_G[0,2) + m_G[2,n] = n$,
this inequality generalizes the result in \cite{MR3313404} for trees.

Our paper seeks {\em applications} to the inequalities $m_G[0,1) \leq  \dom{G}$ and $\dom{G} \leq m_G [2,n]$.
We also seek insight into the {\em ratios} of these numbers.
In the examples given in \cite{HJT2015},
the numbers $\dom{G}$ and $m_G[0,1)$
were equal or differed by one.
We will see that this does not happen in general.

The remainder of our paper is organized as follows.
We finish this introduction by considering the sharpness of these inequalities.
In the next section we recall the proof of
Theorem~\ref{lemgamma} and modify it to obtain an inequality involving $m_G[2,n]$.
In Section~\ref{sec:apps} we obtain several
new results based on existing
Nordhaus-Gaddum inequalities
and Gallai-type theorems.
One interesting new Nordhaus-Gaddum result is that
for any graph $G$, $m_G[0,1) + m_{\bar{G}}[0,1) \leq  n + 1$
with equality if and only if $G = K_n$ or $G = \bar{K_n}$.
Another interesting result is that a graph must have fewer than $\sqrt{n}$
Laplacian eigenvalues in at least one of the intervals $[0,1)$ or $(n-1,n]$.
In Section~\ref{sec:approx},
using results from the approximation literature,
we explain why we can't expect the quantities
$m_G[0,1)$ or $m_G [2,n]$ to be close to $\dom{G}$.
Using some results on Vizing's conjecture,
we show that $\frac{\dom{G} }{ m_G[0,1) } \notin O(\log n)$.
For  trees, $\dom{T} \leq m_T[2,n] < 2 \dom{T}$.
For $c$-cyclic graphs $G$, $c \geq 1$, $m_G[2,n]  \leq (c+ 1) \dom{G}$.
These results seem interesting
in light of the domination number's
general inapproximability.
In Section~\ref{sec:concluding} we observe that many results
also hold for the signless Laplacian spectrum.

\vspace{.1in}
\paragraph{\bf{Tightness}}
We briefly discuss whether $\dom{G}$ is the natural
graph parameter bounded below by $m_G[0,1)$ and above by $m_G[2,n]$.
For example, one might ask
if there exists a graph parameter $p(G)$ for which
$$
m_{G}[0,1) \leq p(G) \leq \dom{G} .
$$
We considered three well-known graph parameters, each bounded above by $\dom{G}$,
and observed that they are not always bounded below by $m_{G}[0,1)$.
More precisely, while the {\em 2-packing number} $\rho(G)$ (see \cite{MR2014537})
is always at most $\dom{G}$, we can find a graph for which $\rho(G) < m_{G}[0,1)$.
Similar examples can be found
for the {\em fractional domination number} $\fractional{G}$ \cite{Grinstead1989},
and the {\em irredundance number} $\ir(G)$ \cite{Damaschke1991}.
We omit the details.

One can also ask if there exists a graph parameter $q(G)$
for which
$$
\dom{G} \leq q(G) \leq m_G [2,n]
$$
for isolate-free $G$.
Graph parameters $q(G)$ for which $\dom{G} \leq q(G)$
include the {\em independent domination number} $i(G)$,
the {\em edge covering number} $\alpha_1 (G)$,
and
the matching number
$\beta_1 (G)$.
In the first two cases we can provide counter examples to show
they are not necessarily bounded above by $ m_G [2,n]$.
Interestingly, we will see that
$\dom{G} \leq \beta_1 (G) \leq m_G [2,n]$,
when $G$ is isolate-free.

\section{Upper bound for $\dom{G}$}
\label{sec:upperbound}
In this section we show how to modify
the proof of Theorem~\ref{lemgamma}
to obtain a new inequality.
For convenience, we recall the facts used to prove
Theorem~\ref{lemgamma}.
Proofs or references can be found in \cite{HJT2015}.
In this paper, a {\em star} $S_n$ is the complete bipartite graph $K_{1, n-1}$, and $n \geq 2$.

\begin{Lem}
\label{lemstar}
The star $S_n$ on $n$ vertices has Laplacian
spectrum $0, 1^{n-2}, n$.
\end{Lem}

\begin{Lem}
\label{lemlaplaciansum}
For graphs $G_1 = (V, E_1)$ and $G_2 = (V,E_2)$ where $E_1 \cap E_2 = \emptyset$,
and $G = (V, E_1 \cup E_2)$, we have $L_G = L_{G_1} + L_{G_2}$.
\end{Lem}

Let $\lambda_i (A)$ denote the $i$-th largest eigenvalue of a Hermitian matrix $A$.
\begin{Lem}
\label{lemcourantweyl}
If $A$ and $B$ are Hermitian matrices of order $n$, and $B$ is positive semi-definite,
then $\lambda_i (A + B) \geq \lambda_i (A)$,
for $1 \leq i \leq n$.
\end{Lem}

\begin{Lem}
\label{lemsubgraph}
Let $G = (V,E)$ and $H = (V,F)$ be graphs with $F \subseteq E$.  Then
\begin{enumerate}
\item for all $i$, $\mu_i (H) \leq \mu_i (G)$;
\item for any $a$, $m_H[0,a) \geq m_G[0,a)$;
\item for any $a$, $m_H [a,n] \leq m_G [a,n]$.
\end{enumerate}
\end{Lem}

Let $S$ be a set of vertices, and $u \in S$.
A vertex $v \in V - S$ is an {\em external private neighbor} of
$u$ (with respect to $S$) if $N(v) \cap S = \{u\}$.
That is, $v \in V - S$ is a neighbor of $u$, but not
a neighbor of any other member of $S$.

\begin{Lem}
[\cite{MR542545}]
\label{lembollabascockayne}
Any graph without isolated vertices has a minimum dominating set
in which every member has an external private neighbor.
\end{Lem}

We will say that $G$ has a {\em star forest} $F = ( S_{n_1}, \ldots , S_{n_k} )$,
if there exists a sequence of pairwise vertex-disjoint subgraphs $H_i$ of $G$,
with $H_i \simeq S_{n_i}$, for all $i$, $1 \leq i \leq k$.
We emphasize that stars have order $n \geq 2$.

\begin{Lem}
\label{lemspanningforest}
Any isolate-free graph $G = (V,E)$ with domination number $\gamma$ has
a star forest $F = ( S_{n_1}, \ldots , S_{n_\gamma})$ such that every $v \in V$ belongs
to exactly one star, and the centers of the stars form a minimum dominating set.
\end{Lem}

Theorem~\ref{lemgamma}
is a spectral lower bound for $\dom{G}$.
The key to its proof
was to take the star forest
that cover all vertices,
$$
F = (S_{n_1}, S_{n_2}, \ldots, S_{n_{\dom{G}} }),
$$
guaranteed by Lemma~\ref{lemspanningforest}.
By Lemma~\ref{lemstar} $m_{S_{n_i}}[0,1) = 1$,
and so $m_F [0,1) = \dom{G}$.
By part (2) of Lemma~\ref{lemsubgraph}
we have $\dom{G} = m_F [0,1) \geq m_G [0,1)$.

If instead of counting the {\em smallest} eigenvalue
in each star we count the {\em largest},
we can also obtain a spectral {\em upper} bound for $\dom{G}$.
Assume that $G$ is isolate-free.
In the construction of $F$,
each star $S_k$ contains $k \geq 2$ vertices.
When $k = 2$, the star has eigenvalues $0, 2$.
When $k \geq 3$, the star has eigenvalues $0, 1^{k-2}, k$.
So $m_{S_{n_i}}[2,n] = 1$ for all $i$.
Since these are disjoint stars, $m_F [2,n] = \dom{G}$.
By Lemma~\ref{lemsubgraph},
part (3),
$m_F [2,n] \leq  m_G [2,n]$.
We conclude that
\begin{Thr}
\label{gammatwo}
If $G$ is an isolate-free graph, then $\dom{G} \leq  m_G [2,n]$.
\end{Thr}

We will use some ideas from our proof of Theorem~\ref{gammatwo}
to establish Theorem~\ref{thr:treeratiob} and Theorem~\ref{thr:ccylicratiob},
later in Section~\ref{sec:approx}.
However, there is actually an alternative and simpler proof to
Theorem~\ref{gammatwo} which we sketch.
Recall that the {\em matching number} $\beta_1(G)$,
is the size of a largest set of independent edges in $G$.
We first claim that
$\beta_1(G) \leq m_G[2,n]$
{\em for any graph} $G$.
To see this, let $F$ be the subgraph of $G$ consisting of
$\beta_1(G)$ disjoint $K_2$'s and $n - 2 \beta_1(G)$ isolated vertices.
Then $m_F[2,n] = \beta_1(G)$.
By part (3) of Lemma~\ref{lemsubgraph}, we must have $\beta_1(G) = m_F[2,n] \leq m_G[2,n]$.
Finally, it is known \cite{HHS1998}
that if $G$ is isolate-free then $\dom{G} \leq  \beta_1(G)$, and so Theorem~\ref{gammatwo} follows.

A connection between $\beta_1(G)$ and the number of Laplacian eigenvalues
{\em strictly} greater than two was shown in 2001 by Ming and Wang \cite{MingWang2001}.
They proved that if $G$ is connected and $n > 2\beta_1(G)$,
then $\beta_1(G) \leq m_G(2,n]$.

Theorem~\ref{gammatwo} strengthens a recent result
by Zhou, Zhou and Du \cite{MR3313404}
which says that for trees $T$, $m_T [0,2) \leq  n - \dom{T}$.
Note that Theorem~\ref{gammatwo} requires $G$ be isolate-free
while Theorem~\ref{lemgamma} does not.
This happens because isolates in Theorem~\ref{lemgamma}
can be disregarded as they
increase both sides of the inequality by one.
In Theorem~\ref{gammatwo} an isolate increases one side
of the inequality but not the other.
Theorem~\ref{lemgamma} and Theorem~\ref{gammatwo} imply
\begin{Cor}
\label{upperlower}
If $G$ is isolate-free then $m_G[0,1) \leq  \dom{G} \leq m_G [2,n]$.
\end{Cor}

It seems interesting in its own right that
\begin{Cor}
If $G$ is isolate-free, then $m_G[0,1) \leq  m_G [2,n]$.
\end{Cor}

When combined with a known lower bound on
$m_T[0,2)$ for trees, Theorem~\ref{lemgamma} implies something
interesting about the interval $[1,2)$.
\begin{Cor}
If $T$ is a tree, then $m_T[1,2) \geq \theratio{n}{2} - \dom{T}$.
\end{Cor}
\begin{proof}
We have
\begin{eqnarray*}
m_T[1,2) & = & m_T [0,2) - m_T[0,1) \\
         & \geq & \theratio{n}{2} - m_T[0,1) \\
         & \geq & \theratio{n}{2} - \dom{T}
\end{eqnarray*}
The first inequality follows by the bound $m_T[0,2) \geq \theratio{n}{2}$
for trees given in \cite[Thr. 4.1]{Braga2013}.
The second inequality follows from Theorem~\ref{lemgamma}.
\end{proof}

\section{Applications}
\label{sec:apps}
Recall that the {\em distance} between vertices $u$ and $v$ is the
number of edges in a shortest path between them,
and the graph's {\em diameter}, $\diam (G)$, is the greatest distance
between any two vertices.
It is known \cite{GroneMerrisSunder} that for trees $T$,
$\floorratio { \diam (T) }{2}$
is a lower bound for both $m_T(0,2)$ and $m_T(2,n]$.
For $G$ connected, it is also known \cite{HHS1998} that
$\frac{ 1 + \diam(G) }{3} \leq \dom{G}$,
so Theorem~\ref{gammatwo} implies
\begin{Cor}
\label{diam-upper}
For connected graphs $G$, $\frac{ 1 + \diam(G) }{3} \leq m_G [2,n]$.
\end{Cor}

\paragraph{\bf{Nordhaus-Gaddum inequalities}}
A {\em Nordhaus-Gaddum} inequality is a bound on the sum or product
of a parameter for a graph $G$ and its complement $\bar{G}$.
For an overview of Nordhaus-Gaddum inequalities for domination-related parameters
we refer to Chapter 10 in \cite{HHS1998}.
A result of Jaeger and Payan \cite{Jaeger72} says that if $G$ is a graph
then
\begin{eqnarray}
\dom{G} + \dom{\bar{G}} & \leq & n + 1  \label{NGaa} \\
\dom{G} \dom{\bar{G}} & \leq & n
\label{NGbb}
\end{eqnarray}
and these bounds are tight.
The following theorem
by Cockayne and Hedetniemi
characterizes when
equality occurs in (\ref{NGaa}).
\begin{Thr}[\cite{Cockayne77}]
\label{thr:cock-hedet}
For any graph $G$, $\dom{G} + \dom{\bar{G}}  \leq  n + 1$ with equality if and only if
$G = K_n$ or $G = \bar{K_n}$.
\end{Thr}
We can use this to obtain the following:
\begin{Thr}
\label{thr:our-nord}
For any graph $G$, $m_G[0,1) + m_{\bar{G}}[0,1) \leq  n + 1$
with equality if and only if $G = K_n$ or $G = \bar{K_n}$.
\end{Thr}
\begin{proof}
From Theorem~\ref{lemgamma} and (\ref{NGaa}) we must have
\begin{equation}
\label{eq:sandwich}
m_G[0,1) + m_{\bar{G}}[0,1) \leq  \dom{G} + \dom{\bar{G}} \leq  n + 1
\end{equation}
for any $G$.
Since $m_{K_n} [0,1) = 1$ and $m_{ \bar{ K_n }} [0,1) = n$,
we must have equality
if $G = K_n$ or $G = \bar{K_n}$.
Conversely if $m_G[0,1) + m_{\bar{G}}[0,1) =  n + 1$, then (\ref{eq:sandwich}) forces
$\dom{G} + \dom{\bar{G}} = n + 1$.
By Theorem~\ref{thr:cock-hedet} it follows that $G = K_n$ or $G = \bar{K_n}$.
\end{proof}

From Theorem~\ref{lemgamma} and (\ref{NGbb}) we also have
\begin{Thr}
\label{thr:NGbb}
For any graph $G$, $m_G[0,1) \cdot m_{\bar{G}}[0,1) \leq  n$.
\end{Thr}

Recall \cite[Theorem 3.6]{mohar}
that if $G$ has Laplacian eigenvalues
$$
0 = \mu_1 \leq \mu_2 \leq \ldots \leq \mu_n
$$
then
the Laplacian eigenvalues
of $\bar{G}$ are:
$$
0, ~~ n - \mu_n, ~~ n - \mu_{n-1}, ~~~ \ldots,  ~~ n - \mu_2
$$
It follows that $m_{\bar{G}}[0,1) = m_G(n-1,n] + 1$.
Then from Theorem~\ref{thr:NGbb}
\begin{eqnarray*}
m_G[0,1) \cdot m_G(n-1,n] & < &   m_G[0,1) \cdot  ( m_G(n-1,n] + 1 )   = \\
m_G[0,1) \cdot  m_{\bar{G}}[0,1)  & \leq   & n .
\end{eqnarray*}
We have
\begin{Thr}
\label{thr:intproducts}
For any graph $G$, $m_G[0,1)  \cdot m_G(n-1,n] < n$.
\end{Thr}
We conclude that any graph of order $n$ must have fewer than $\sqrt{n}$
Laplacian eigenvalues in at least one of the intervals $[0,1)$ or $(n-1,n]$.

\vspace{.1in}
\paragraph{\bf{Gallai-type theorems}}
A {\em Gallai-type} theorem has the form $x(G) + y(G) = n$ where $x(G)$ and $y(G)$
are graph parameters.
There are exactly $n$
Laplacian eigenvalues,
so the equation
\begin{equation}
\label{eq:gallai-eigen}
m_G[0,1) + m_G [1,n] = n
\end{equation}
can be regarded as a trivial Gallai-type theorem.
A {\em spanning forest} of a graph $G$ is a spanning
subgraph which contains no cycles.
Let $\varepsilon(G)$ denote the maximum number of pendant edges
in a spanning forest of $G$.
\begin{Thr} [ Nieminen \cite{Nieminen74} ]
\label{thr:Nieminen}
For any graph $G$, $\dom{G} + \varepsilon(G) = n$.
\end{Thr}

\begin{Cor}
\label{greaterthanone}
For any graph $G$,
$\varepsilon (G) \leq m_G [1,n]$.
\end{Cor}
\begin{proof}
From Theorem~\ref{lemgamma}
and (\ref{eq:gallai-eigen})
we know that
\begin{equation}
\label{ineq:dom1n}
n - \dom{G} \leq m_G [1,n]
\end{equation}
the left side being $\varepsilon (G)$
by Theorem~\ref{thr:Nieminen}.
\end{proof}

\begin{Cor}
$\dom{G} = m_G[0,1)$ if and only if $\varepsilon (G) = m_G [1,n]$.
\end{Cor}
\begin{proof}
This follows from (\ref{eq:gallai-eigen}) and Theorem~\ref{thr:Nieminen}.
\end{proof}

Berge \cite{Berge73} gives an early bound for $\dom{G}$:
\begin{equation}
\label{berge}
\dom{G} + \Delta (G) \leq n
\end{equation}
where $\Delta$ denotes the maximum vertex degree.
In \cite{DomkeDunbarMarkus97} the authors study when equality
in (\ref{berge}) occurs.
Combining (\ref{ineq:dom1n})
and (\ref{berge}) give
\begin{Thr}
\label{thr:maxdegree}
For any graph $G$, $m_G [1,n] \geq \Delta(G)$.
\end{Thr}
As a simple application to Theorem~\ref{thr:maxdegree},
suppose we are given a list $\sigma$
$$
0  = \mu_n \leq  \mu_{n-1} \leq \ldots \leq \mu_1
$$
of non-negative numbers
and wish to know if there is a graph $G$ whose  Laplacian spectrum is $\sigma$.
Then Theorem~\ref{thr:maxdegree} imposes a {\em necessary} condition on $G$.
Let $B = | \{ i : \mu_i \geq 1 \} |$.
Any graph $G$ such that $\spec(G) = \sigma$ must have vertices whose
degrees are bounded by $B$.

\section{Approximating $\dom{G}$}
\label{sec:approx}
In this section we explain why it is hard to approximate $\dom{G}$
with a polynomial computable spectral quantity
of the form $m_G[a,b]$.  We show that $m_G[0,1)$ and $m_G[2,n]$
do not even achieve logarithmic approximation ratios.
Yet, for certain classes of graphs such as
trees and $c$-cyclic graphs, $\frac{m_G [2,n]} { \dom{G} }$ is bounded by a constant.

\vspace{.1in}
\paragraph{\bf{Inapproximability}}
It is well-known that the decision problem
DOMINATING SET
is NP-complete \cite{GareyandJohnson79},
even for planar graphs.
In the approximation algorithm literature
the problem is classified as class II in the taxonomy
of NP-complete problems given in \cite{AroraLund}.
Roughly speaking, this means that approximating with better
than a logarithmic ratio is hard.
A problem is called {\em quasi-NP-hard} if a polynomial-time algorithm
for it could be used to solve all NP problems in time $2^{\poly(\log n)}$.
Thus the notion is slightly weaker than NP-hard.

Lund and Yannakakis \cite[Thr. 3.6]{Lund1994} showed that it is quasi-NP-hard
to compute a polynomial-time function $f(G) \geq \dom{G}$ for which
$$
\frac{f(G)}{\dom{G}} \leq c \log_2 n
$$
when $0 < c < \frac{1}{4}$.
Letting $g(G) = \frac { f(G) }{ c \log_2 n}$, we see this is equivalent to
computing a polynomial time $g(G) \leq \dom{G}$ for which
$$
\frac{\dom{G}}{g(G)} \leq c \log_2 n .
$$

Good approximations of $\dom{G}$ do exist.
The fractional domination number $\fractional{G}$
can be computed in polynomial time using linear programming.
Given a vertex ordering,
we can compute in polynomial time
an approximation $\gamma_g(G)$
for $\dom{G}$ using the greedy domination algorithm.
Clearly for any graph $G$,
$$
\fractional{G} \leq \dom{G} \leq \gamma_g(G).
$$
In \cite{CGH2015} Chappell, Gimbel and Hartman
proved that $\frac{ \gamma_g(G) }{ \fractional{G} }$
is in  $O(\log n)$.
It follows that
both
$\frac{ \gamma_g(G) }{ \dom{G} }$
and
$\frac{ \dom{G} }{ \fractional{G} }$
must also be in $O(\log n)$.
Note this result does not contradict that of Lund and Yannakakis,
provided the constants of proportionality are sufficiently large.

\begin{figure}[h!]
\begin{center}
\begin{tikzpicture}
[scale=1,auto=left,every node/.style={circle,scale=0.8}]
\node at (0,0) [circle,draw](1) {};
\node at (.5,0) [circle,draw](2) {};
\node at (1,0) [circle,draw](3) {};
\node at (1.5,0) [circle,draw](4) {};
\node at (2,0) [circle,draw](5) {};
\node at (2.5,0) [circle,draw](6) {};
\node at (3,0) [circle,draw](101) {};
\node at (3.5,0) [circle,draw](102) {};
\node at (4,0) [circle,draw](103) {};
\node at (4.5,0) [circle,draw](104) {};
\node at (5,0) [circle,draw](105) {};
\node at (5.5,0) [circle,draw](106) {};
\node at (6,0) [circle,draw](201) {};
\node at (6.5,0) [circle,draw](202) {};
\node at (7,0) [circle,draw](203) {};
\node at (7.5,0) [circle,draw](204) {};
\node at (8,0) [circle,draw](205) {};
\node at (8.5,0) [circle,draw](206) {};
\node at (9,0) [circle,draw](301) {};
\node at (9.5,0) [circle,draw](302) {};
\node at (10,0) [circle,draw](303) {};
\node at (10.5,0) [circle,draw](304) {};
\node at (11,0) [circle,draw](305) {};
\node at (11.5,0) [circle,draw](306) {};
\node at (0,-1) [circle,draw](7) {};
\node at (.5,-1) [circle,draw](8) {};
\node at (1.0,-1) [circle,draw](9) {};
\node at (1.5,-1) [circle,draw](10) {};
\node at (2,-1) [circle,draw](11) {};
\node at (2.5,-1) [circle,draw](12) {};
\node at (3,-1) [circle,draw](107) {};
\node at (3.5,-1) [circle,draw](108) {};
\node at (4.0,-1) [circle,draw](109) {};
\node at (4.5,-1) [circle,draw](110) {};
\node at (5,-1) [circle,draw](111) {};
\node at (5.5,-1) [circle,draw](112) {};
\node at (6,-1) [circle,draw](207) {};
\node at (6.5,-1) [circle,draw](208) {};
\node at (7.0,-1) [circle,draw](209) {};
\node at (7.5,-1) [circle,draw](210) {};
\node at (8,-1) [circle,draw](211) {};
\node at (8.5,-1) [circle,draw](212) {};
\node at (9,-1) [circle,draw](307) {};
\node at (9.5,-1) [circle,draw](308) {};
\node at (10.0,-1) [circle,draw](309) {};
\node at (10.5,-1) [circle,draw](310) {};
\node at (11,-1) [circle,draw](311) {};
\node at (11.5,-1) [circle,draw](312) {};
\node at (.25,1) [circle,draw](13) {};
\node at (1.25,1) [circle,draw](14) {};
\node at (2.25,1) [circle,draw](15) {};
\node at (3.25,1) [circle,draw](113) {};
\node at (4.25,1) [circle,draw](114) {};
\node at (5.25,1) [circle,draw](115) {};
\node at (6.25,1) [circle,draw](213) {};
\node at (7.25,1) [circle,draw](214) {};
\node at (8.25,1) [circle,draw](215) {};
\node at (9.25,1) [circle,draw](313) {};
\node at (10.25,1) [circle,draw](314) {};
\node at (11.25,1) [circle,draw](315) {};
\node at (1.25,2) [circle,draw](16) {};
\node at (4.25,2) [circle,draw](116) {};
\node at (7.25,2) [circle,draw](216) {};
\node at (10.25,2) [circle,draw](316) {};
\node at (5.75,3) [circle,draw](1000) {};
\draw (1000.south) -- (16.north);
\draw (1000.south) -- (116.north);
\draw (1000.south) -- (216.north);
\draw (1000.south) -- (316.north);
\draw (1.south) -- (7.north);
\draw (2.south) -- (8.north);
\draw (4.south) -- (10.north);
\draw (3.south) -- (9.north);
\draw (5.south) -- (11.north);
\draw (6.south) -- (12.north);
\draw (13.south) -- (1.north);
\draw (13.south) -- (2.north);
\draw (14.south) -- (3.north);
\draw (14.south) -- (4.north);
\draw (15.south) -- (5.north);
\draw (15.south) -- (6.north);
\draw (16.west) -- (13.north);
\draw (16.south) -- (14.north);
\draw (16.east) -- (15.north);
\draw (101.south) -- (107.north);
\draw (102.south) -- (108.north);
\draw (104.south) -- (110.north);
\draw (103.south) -- (109.north);
\draw (105.south) -- (111.north);
\draw (106.south) -- (112.north);
\draw (113.south) -- (101.north);
\draw (113.south) -- (102.north);
\draw (114.south) -- (103.north);
\draw (114.south) -- (104.north);
\draw (115.south) -- (105.north);
\draw (115.south) -- (106.north);
\draw (116.west) -- (113.north);
\draw (116.south) -- (114.north);
\draw (116.east) -- (115.north);
\draw (201.south) -- (207.north);
\draw (202.south) -- (208.north);
\draw (204.south) -- (210.north);
\draw (203.south) -- (209.north);
\draw (205.south) -- (211.north);
\draw (206.south) -- (212.north);
\draw (213.south) -- (201.north);
\draw (213.south) -- (202.north);
\draw (214.south) -- (203.north);
\draw (214.south) -- (204.north);
\draw (215.south) -- (205.north);
\draw (215.south) -- (206.north);
\draw (216.west) -- (213.north);
\draw (216.south) -- (214.north);
\draw (216.east) -- (215.north);
\draw (301.south) -- (307.north);
\draw (302.south) -- (308.north);
\draw (304.south) -- (310.north);
\draw (303.south) -- (309.north);
\draw (305.south) -- (311.north);
\draw (306.south) -- (312.north);
\draw (313.south) -- (301.north);
\draw (313.south) -- (302.north);
\draw (314.south) -- (303.north);
\draw (314.south) -- (304.north);
\draw (315.south) -- (305.north);
\draw (315.south) -- (306.north);
\draw (316.west) -- (313.north);
\draw (316.south) -- (314.north);
\draw (316.east) -- (315.north);
%
\end{tikzpicture}
\caption{ \label{treefig} $m_T[0,1) = 24 < \dom{T} = 25$ }
\end{center}
\end{figure}

\vspace{.1in}
\paragraph{\bf{Example}}
We now construct an infinite sequence of graphs for which the ratio
$\frac{\dom{G} }{ m_G[0,1)  } \not\in O(\log n)$.
Our construction uses the tree $T$ of order $n = 65$, shown in Figure~\ref{treefig}.
It is known \cite{HJT2015} that $m_T[0,1) = 24$  and $\dom{T} = 25$.

Recall that the {\em Cartesian product} $G \times H$ of two graphs $G = (V,E)$ and $H = (W,F)$ is the graph
with vertex set $V \times W$ for which $(v_1, w_1)$ and $(v_2, w_2)$
are adjacent if and only if $v_1 = v_2$ and $w_1 w_2 \in F$ or $w_1 = w_2$ and $v_1 v_2 \in E$.

In 1968 V. G. Vizing conjectured \cite{MR0240000}
that for all graphs $G$ and $H$,
\begin{equation}
\label{vizingeq}
\dom{G} \cdot \dom{H} \leq \dom{G \times H}
\end{equation}

While this currently remains
an open problem, many partial results exist.
We say that $G$ {\em satisfies Vizing's conjecture} if (\ref{vizingeq})
holds for all graphs $H$.
Many classes of graphs are known to satisfy Vizing's conjecture.
\begin{Lem}[ Theorem 8.2, \cite{MR2864622}]
\label{lemtreesVizings}
All trees satisfy Vizing's conjecture.
\end{Lem}

It is easy to show that the Cartesian product is an associative operation.
Let $G^k$ denote the Cartesian product $G \times \ldots \times G$ of $k$ copies of $G$.

\begin{Lem}
\label{lemVizingproduct}
If $G$ satisfies Vizing's conjecture, then $\dom{G}^k \leq \dom{G^k}$.
\end{Lem}
\begin{proof}
By induction on $k$, the case for $k = 1$ being trivial.
Assume that $\dom{G}^k \leq \dom{G^k}$.
Using the induction assumption, the fact that $G$ satisfies Vizing's conjecture,
and the associativity of $\times$, we have
$$
\dom{G}^{k+1} = \dom{G} \dom{G}^k \leq \dom{G} \dom{G^k} \leq \dom{G \times G^k} = \dom{G^{k+1}}
$$
completing the proof.
\end{proof}

The following is well-known (See, for example, \cite[Thr. 3.5]{mohar}).
\begin{Lem}
\label{lemproduct}
Let $G$ and $H$ be graphs with Laplacian spectra
$$
0  = \mu_n \leq  \mu_{n-1} \leq \ldots \leq \mu_1
$$
and
$$
0  = \mu^{\prime}_m \leq  \mu^{\prime}_{m-1} \leq \ldots \leq \mu^{\prime}_1
$$
respectively.
Then the Laplacian spectrum of $G \times H$ is
$$
\{ \mu_i + \mu^{\prime}_j | 1 \leq i \leq n, 1 \leq j \leq m \} .
$$
\end{Lem}

\begin{Lem}
\label{leminequal}
For any graphs $G$ and $H$, $m_{G \times H} [0,1) \leq m_G [0,1) \cdot m_H [0,1)$.
\end{Lem}
\begin{proof}
By Lemma~\ref{lemproduct},
Laplacian eigenvalues of $G \times H$
are of the form
$\mu_i + \mu^{\prime}_j$,
where $\mu_i$ and $\mu^{\prime}_j$ are eigenvalues of $G$ and $H$ respectively.
A necessary condition for
$\mu_i + \mu^{\prime}_j < 1$
is that $\mu_i < 1$ and $\mu^{\prime}_j < 1$.
There are at most $m_G [0,1) \cdot m_H [0,1)$ such pairs.
\end{proof}

\begin{Lem}
\label{leminequalA}
For any graph $G$ and any $k \geq 1$, $m_{G^k} [0,1) \leq m_G [0,1)^k $.
\end{Lem}
\begin{proof}
The case $k=1$ is trivial, and $k=2$ is handled by
Lemma~\ref{leminequal}.
Assume $m_{G^k} [0,1) \leq m_G [0,1)^k $.
Then using Lemma~\ref{leminequal} and the induction assumption,
we have:
$$
m_{G^{k+1}} [0,1) = m_{G \times G^{k}} [0,1) \leq m_G [0,1) \cdot m_{G^k}[0,1) \leq m_G [0,1) \cdot m_G [0,1)^k .
$$
The right side is $m_G[0,1)^{k+1}$ completing the induction.
\end{proof}

Let $T$ be the tree of order $65$ in Figure~\ref{treefig}
for which
\begin{equation}
\label{eq:tree-facts}
m_T[0,1) = 24  \mbox{ and } \dom{T} = 25 .
\end{equation}
We claim that for all $k \geq 1$
\begin{equation}
\label{longinequal}
m_{T^k} [0,1) \leq m_T [0,1)^k \leq \dom{T}^k \leq \dom{T^k}
\end{equation}
The first inequality follows by Lemma~\ref{leminequalA},
and the second inequality follows by Theorem~\ref{lemgamma}.
The third inequality follows by
Lemma~\ref{lemtreesVizings} and Lemma~\ref{lemVizingproduct}.

\begin{Thr}
\label{thr:lowerratio}
There exists a sequence of graphs $G_k$
with $\frac{ \dom{G_k} }{ m_{G_k} [0,1) }  \not\in O(\log n)$.
\end{Thr}
\begin{proof}
We let $G_k = T^k$.
Using $n = 65^k$,  (\ref{eq:tree-facts}) and (\ref{longinequal}) we have
$$
\frac{ \dom{T^k} }{ m_{T^k} [0,1) } \geq
\frac{ \dom{T}^k }{ m_T [0,1)^k } =
\left ( \frac{25}{24} \right )^k =
\left ( \frac{25}{24} \right )^ {\log_{65} n} =
n^ {\log_{65} \frac{25}{24}} = n^{.009779} .
$$
\end{proof}

\paragraph{\bf{Ratios for certain classes}}
Consider the two approximation ratios:
\begin{equation}
\frac{\dom{G} }{ m_G[0,1) } \label{ratiob} \\
\end{equation}

\begin{equation}
\frac{ m_G[2,n] }{\dom{G}} \label{ratioa}
\end{equation}

Both ratios can get arbitrarily large.
By Theorem~\ref{thr:lowerratio}
the first of these ratios
is not bounded by $\log (n)$.
The second ratio also gets arbitrarily large.
When $G = K_n$ is the complete graph,
we see that ratio (\ref{ratioa}) is $n-1$.

Consider (\ref{ratioa}) for paths $P_n$.
It is well-known that $\dom{P_n} = \theratio{n}{3}$.
By Thr. 4.1 in \cite{Braga2013}
we also know $m_{P_n}[2,n] \leq \floorratio{n}{2}$,
and so (\ref{ratioa}) is at most $\frac{3}{2}$.
Using ideas from Section~\ref{sec:upperbound},
we show that for {\em all trees}  ratio (\ref{ratioa}) is less than two.

\begin{Lem}
\label{lem:addedge}
Let $G$ be a graph on $n$ vertices and $m < {n \choose 2}$ edges,
and let
$G^\prime$
be the graph obtained by adding an edge.
Then for any $a \geq 0$,
$$
m_G[a,n] \leq m_{G^\prime}[a,n] \leq m_G[a,n] + 1 .
$$
\end{Lem}
\begin{proof}
Let
$0 = \mu_n \leq \ldots \leq \mu_2 \leq \mu_1$
and
$0 = \mu^\prime_n \leq \ldots \leq \mu^\prime_2 \leq \mu^\prime_1$
be the respective Laplacian spectra of $G$ and $G^\prime$.
By the well-know interlacing theorem \cite[Thr. 2.4]{MR2494114}
for Laplacian eigenvalues we know
$$
0 = \mu_n = \mu^\prime_n \leq \ldots \leq \mu_k \leq \mu^\prime_k \leq \ldots \leq
\mu_2 \leq \mu^\prime_2 \leq \mu_1 \leq \mu^\prime_1
$$
If $a = 0$, then $m_G[a,n] = m_{G^\prime}[a,n] = n$.
If $\mu_1 < a$ then $m_G[a,n] =  m_{G^\prime}[a,n] = 0$.
We may assume that $0 < a \leq \mu_1$.
Choose $k$ to be the largest index for which $a \leq \mu_k$.
Then $\mu_{k+1} < a \leq \mu_k$.
There is a single eigenvalue of $G^\prime$,
namely $\mu^\prime_{k+1}$ in $[ \mu_{k+1}, \mu_k ] $.
If $\mu^\prime_{k+1} \leq a$,
then $m_{G^\prime}[a,n] = m_G[a,n] + 1$.
Otherwise, $m_{G^\prime}[a,n] = m_G[a,n]$.
\end{proof}

\begin{Thr}
\label{thr:treeratiob}
 If $T$ is a tree, then $1 \leq \frac{ m_T[2,n] }{ \dom{T} } < 2$.
\end{Thr}
\begin{proof}
Let $F = ( S_{n_1}, \ldots , S_{n_\gamma})$ be the star forest guaranteed by
Lemma~\ref{lemspanningforest}.
Then $m_F [2,n]$ is exactly $\dom{T}$.
Starting with $F$, we can construct
$T$ by {\em adding} $\dom{T} - 1$ edges.
By Lemma~\ref{lem:addedge} the addition of each edge
can increase $m_T[2,n]$ by at most one.
Therefore
$$
m_F[2,n] \leq m_T[2,n] \leq m_F[2,n] + \dom{T} - 1 .
$$
But the right side is $2 \dom{T} - 1$
and the theorem follows.
\end{proof}

A connected graph having $n - 1 + c$ edges is called {\em c-cyclic}.
We can generalize Theorem~\ref{thr:treeratiob} as follows.
\begin{Thr}
\label{thr:ccylicratiob}
If $G$ is $c$-cyclic, $c \geq 1$, then $1 \leq \frac{ m_G[2,n] }{ \dom{G} } \leq c+ 1$.
\end{Thr}
\begin{proof}
Let $F = ( S_{n_1}, \ldots , S_{n_\gamma})$ be the star forest in $G$
from Lemma~\ref{lemspanningforest}.
Then we may select $\dom{G} - 1$ additional edges to form a
spanning tree $T$.
Since $T$ has $n-1$ edges, there must be $c$ remaining edges.
Therefore $G$ can be constructed from $F$ by adding $\dom{G} - 1 + c$ edges.
By Lemma~\ref{lem:addedge}
$$
 m_G[2,n] \leq m_F[2,n] + \dom{G} - 1  + c = 2 \dom{G} + c - 1,
$$
or
$$
\frac{ m_G[2,n] }{ \dom{G} } \leq 2 + \frac{c-1}{\dom{G}}  \leq 2 + c - 1 ,
$$
the last inequality holding because $c \geq 1$ and $\dom{G} \geq 1$.
\end{proof}

Let us now consider ratio (\ref{ratiob}) for trees.
For the tree in Figure~\ref{treefig}, the ratio (\ref{ratiob}) is $\frac{25}{24}$.
It is possible to generalize this example.  We construct the tree $T_k$
on $65k + 1$ vertices by taking $k$ copies of this tree, and adjoining the root to each copy.
Using the algorithm in \cite{Braga2013}, it is straightforward to determine
that $ m_{T_k} [0,1) = 24k$.
Using the domination algorithm in \cite{CGH1975}
it can be shown that $\dom{T_k} = 25k$.
Thus, the difference between $\dom{T_k} - m_{T_k} [0,1)$ grows arbitrarily large.
However, the
ratio (\ref{ratiob}) remains at $\frac{25}{24}$.
In all known examples of trees
ratio (\ref{ratiob})
is either $1$ or $\frac{25}{24}$,
and it is tempting to conjecture that
the ratio is bounded by a constant for trees.

\section{Concluding remarks}
\label{sec:concluding}
Many of the results of this paper also apply to the {\em signless Laplacian spectrum}.
For example, if we let $m_{G}^+ I$ denote the number of
signless Laplacian eigenvalues of $G$ in $I$, then
Theorem~\ref{lemgamma} and Theorem~\ref{gammatwo}
are also true if we replace $m_G$ with $m_{G}^+$.

We conclude by suggesting two problems for further study.
First, characterize those graphs $G$ for which $m_G[0,1) = \dom{G}$.
Second, determine if $\frac{\dom{T}} { m_T[0,1) }$ bounded by a constant for trees $T$.

\bibliographystyle{amsplain}
\bibliography{note}          

\end{document}